\documentclass[letterpaper]{article}

\usepackage{amsmath}
\usepackage{amssymb}
\usepackage{amsthm}

\usepackage{times}

\newtheorem{lemma}{Lemma}[section]
\newtheorem{theorem}[lemma]{Theorem}
\newtheorem{definition}[lemma]{Definition}
\newtheorem{corollary}[lemma]{Corollary}

\author{Henry Towsner}
\title{A General Correspondence between Averages and Integrals}
\date{\today}

\begin{document}
\bibliographystyle{apalike}
\maketitle

\begin{abstract}
  Recent work has generalized the Furstenberg correspondence between sets of integers and dynamical systems to versions which involve sequences of finite graphs or sequences of $L^\infty$ functions.  We give a unified version of the theorem subsuming all these generalizations.
\end{abstract}

\section{Introduction}

The Furstenberg correspondence \cite{FurstenbergBook} was originally developed in order to use ergodic methods to prove Szemer\'edi's Theorem, that every set of integers with positive upper Banach density contains arbitrarily long arithmetic progressions.  The correspondence is based on the following theorem:
\begin{theorem}
  Let $E\subseteq \mathbb{Z}$ with positive upper Banach density be given.  Then there is a dynamical system $(Y,\mathcal{B},\mu,T)$ and a set $A\in\mathcal{B}$ with $\mu(A)>0$ such that for any finite set of integers $U$, the upper Banach density of $\bigcap_{n\in U} (E-n)$ is at least $\mu(\bigcap_{n\in U}T^n A)$.
\end{theorem}

By a dynamical system $(Y,\mathcal{B},\mu,\{T_g\}_{g\in G})$, we mean a measure space $(Y,\mathcal{B},\mu)$ together with a collection of measurable, measure-preserving transformations $T_g:Y\rightarrow Y$ indexed by a semigroup $G$ such that $T_g\circ T_h=T_{gh}$.  When $G=\mathbb{Z}$ we often write $(Y,\mathcal{B},\mu,T)$ where $T=T_1$.

\cite{Bergelson} generalizes this to countable amenable groups.
\begin{definition}
  If $\{I_n\}_{n\in\mathbb{N}}$ is a left F\o lner sequence of $G$, for any $E\subseteq G$ define
\[\overline{d}_{\{I_n\}}(E)=\limsup_{n\rightarrow\infty}\frac{|E\cap I_n|}{|I_n|}\]

Say $E$ has positive upper density with respect to $\{I_n\}$ if $\overline{d}_{\{I_n\}}(E)>0$.
\end{definition}

\begin{theorem}
Let $G$ be a countable amenable group and assume that a set $E\subseteq G$ has positive upper density with respect to a left F\o lner sequence $\{I_n\}_{n\in\mathbb{N}}$.  Then there exists a dynamical system $(X,\mathcal{B},\mu,(T_g)_{g\in G})$ and a set $A\in\mathcal{B}$ with $\mu(A)=\overline{d}_{\{I_n\}}(E)$ such that for any $k\in\mathbb{N}$ and $g_1,\ldots,g_k\in G$, one has
\[\overline{d}_{\{I_n\}}(E\cap g_1^{-1}E\cap\cdots\cap g_k^{-1}E)\geq\mu(A\cap T^{-1}_{g_1}A\cdots\cap T^{-1}_{g_k}A)\]
\end{theorem}

Recently there has been a renewed interest, both in Szemer\'edi's Theorem and in the passage between statements about sets of integers and statements about dynamical systems.  In particular, Elek and Szegedy \cite{ES} and Tao \cite{TaoHypergraph} developed new variations on the Furstenberg correspondence for graphs, rather than sets, and used them to provide proofs of the hypergraph regularity lemma using ergodic methods.  Tao also applied the Furstenberg correspondence ``backwards'' to settle an open problem in ergodic theory regarding the convergence of ``diagonal'' ergodic averages of commuting transformations \cite{Tao}.  The author gave a more purely ergodic proof of the same theorem \cite{TowsnerNormConvergence}, but still required two new variations on the Furstenberg correspondence; one in which the set $E$ is replaced by a function $e:\mathbb{Z}\rightarrow [0,1]$, and one in which a discrete average of $L^\infty$ functions on a measure space is replaced by an integral over a product space.

In this paper, we distill the common theme from all these correspondences to give a single general theorem subsuming all these cases.  We give two proofs, first a conventional one in the style of Furstenberg's original proof, and then a proof using techniques of nonstandard analysis.

\section{A Generalized Correspondence}
\begin{definition}
  Let $G$ be a semigroup, and let $\{I_n\}$ be a sequence of finite subsets of $G$.  The sequence $\{I_n\}$ is a (left) F{\o}lner sequence if for each $g\in G$ and each $\epsilon>0$, there is some $M_{g,\epsilon}$ such that whenever $n\geq M_{g,\epsilon}$,
\[\frac{|I_n\bigtriangleup g\cdot I_n|}{|I_n|}<\epsilon\]
\end{definition}

Semigroups for which F{\o}lner sequences exist are precisely the countable amenable semigroups.

\begin{theorem}
  Let $S$ be a countable set and let $G$ be a semigroup acting on $S$.  Let $X$ be a second countable compact space.  Let $E:S\rightarrow X$ be given, and let $\{I_n\}$ be a F{\o}lner sequence of subsets of $G$.

  Then there are a dynamical system $(Y,\mathcal{B},\nu,(T_g)_{g\in G})$ and measurable functions (with respect to the Borel sets generated by the topology on $X$) $\tilde E_{s}:Y\rightarrow X$ for each $s\in S$ such that the following hold:
  \begin{itemize}
  \item For any $g,s$, $\tilde E_{gs}=\tilde E_s\circ T_g$
  \item For any integer $k$, any continuous function $u:X^k\rightarrow \mathbb{R}$, and any finite sequence $s_1,\ldots, s_k$,
\begin{align*}
\liminf_{n\rightarrow\infty}\frac{1}{|I_n|}\sum_{g\in I_n}u(E(g s_1),\ldots,E(gs_k))&\leq\int  u(\tilde E_{s_1},\ldots,\tilde E_{s_k})d\nu\\
&\leq\limsup_{n\rightarrow\infty}\frac{1}{|I_n|}\sum_{g\in I_n}u(E(g s_1),\ldots,E(gs_k))
\end{align*}
\end{itemize}
\label{main}
\end{theorem}

To illustrate this, we give six special cases; the first five have been proven separately, while the sixth, as far as we know, is novel.

\begin{corollary}[\cite{Furst77},\cite{FurstenbergBook}]
  Let $\hat E\subseteq \mathbb{Z}$ with positive upper Banach density be given.  Then there are a dynamical system $(Y,\mathcal{B},\mu,T)$ and a set $A\in\mathcal{B}$ with $\mu(A)>0$ such that for any finite set of integers $U$, the upper Banach density of $\bigcap_{n\in U} (E-n)$ is at least $\mu(\bigcap_{n\in U}T^n A)$.
\end{corollary}
\begin{proof}
  Apply Theorem \ref{main} by letting $G=S=\mathbb{Z}$ and $E:=\chi_{\hat E}:\mathbb{Z}\rightarrow\{0,1\}$.  Let the sequence $\{I_n\}$ be a sequence of intervals witnessing the positive upper Banach density of $E$.

  This gives a dynamical system $(Y,\mathcal{B},\mu,\mathbb{Z})$.  Let $T$ be the action of $1$ on $Y$.  Set $\tilde E:=\tilde E_1$, so $\tilde E_n=\tilde E\circ T^n$ for each $n$.  Then for any $k$, the function $u:\{0,1\}^k\rightarrow\mathbb{R}$ given by
\[u(b_1,\ldots,b_k):=\prod_{i\leq k} b_i\]
is continuous, so the upper Banach density of $\bigcap_{n\in U} (E-n)$ is bounded below by $\int \prod_{n\in U}\tilde E_{n}d\mu=\int \prod_{n\in U}\tilde E\circ T^n d\mu$.
\end{proof}

By similar arguments:
\begin{corollary}[\cite{FKO},\cite{FurstenbergBook}]
  Let $\hat E\subseteq \mathbb{Z}^k$ with positive upper Banach density be given.  Then there are a dynamical system $(Y,\mathcal{B},\mu,T_1,\ldots,T_k)$ and a set $A\in\mathcal{B}$ with $\mu(A)>0$ such that for any finite set of tuples $U$, the upper Banach density of $\bigcap_{\vec n\in U} (E-\vec n)$ is at least $\mu(\bigcap_{\vec n\in U}T_1^{n_1}\cdots T^{n_k}_k A)$.
\end{corollary}

The following corollary is implicit in the more complicated one given in Section 2 of \cite{TowsnerNormConvergence}.
\begin{corollary}
  Let $f:\mathbb{Z}\rightarrow[-1,1]$ be given.  Then there are a dynamical system $(Y,\mathcal{B},\mu,T)$ and a function $F\in L^\infty(Y)$ such that for any finite set of integers $U$,
\[\liminf_{N\rightarrow\infty}\frac{1}{N}\sum_{n=1}^N\prod_{k\in U}f(n+k)\leq\int \prod_{k\in U}F\circ T^kd\mu\leq\limsup_{N\rightarrow\infty}\frac{1}{N}\sum_{n=1}^N\prod_{k\in U}f(n+k)\]
\end{corollary}

\begin{corollary}[\cite{TowsnerNormConvergence}]
  Let $(Z,\mathcal{B},\mu)$ be a separable measure space, and for each $s\leq k$, let $(Z_s,\mathcal{B}_s,\mu_s)$ be a factor.  Let a real number $b$ be given, and for $s\leq k$, let $\{f_{s,n}\}_{i\in\mathbb{N}}$ be a sequence of $L^\infty(Z_s,\mathcal{B}_s,\mu_s)$ functions almost everywhere bounded by $b$.  Let $g$ be a weak limit point of the sequence $\frac{1}{N}\sum_{n=0}^{N-1}\prod_{s\leq k}f_{s,n}$ as $N$ goes to infinity.

Then there are a measure space $(Y,\mathcal{C},\nu)$ and functions $\tilde f_s\in L^\infty(Z_s\times Y)$ such that $\int \prod\tilde f_sd\nu$ is $g$.
\end{corollary}
\begin{proof}
  Let $S$ be $[0,k]\times\mathbb{N}$, let $X$ be the set of $L^\infty(Z,\mathcal{B},\mu)$ functions with norm at most $b$, under the weak${}^*$ topology.  Let $G$ be $\mathbb{N}$, acting on $S$ by $n(s,m):=(s,m+n)$, let $N_t$ be a sequence of integers witnessing that $g$ is a weak limit point, and let $I_t:=[0,N_t]$.  Let $E(i,m):=f_{i,m}$.

  Fix some orthonormal basis $\{g_j\}$ for $L^2(Z)$.  Observe that for each $j$, the function $S_j$ given by $S_j(h):=\sum_{i\leq j}\int hg_jd\mu$ is continuous, and since each $f_{s,n}$ is almost everywhere bounded by $b$, it follows that for every $j$, $||S_j(f_{s,0})||_{L^2(Y)}\leq b$.  Therefore $||\sum_{i\leq j} g_i(z)\int \tilde E_{s,0}(y)g_jd\mu||\leq b$ for every $j$.  Then the infinite sum $\sum_i g_i(z)\int \tilde E_{s,0}(y)g_jd\mu$ is a convergent sum of functions measurable in $L^2(Z\times Y)$, and is therefore measurable.  We may then take this function to be represented by $\tilde f_s(z,y):=\tilde E_{s,0}(y)(z)$.

Next, observe that for each $g_j$, the function $\int\prod\tilde f_sg_jd\mu\times\nu$ is equal to \[\lim_{t\rightarrow\infty}\int\frac{1}{N_t}\sum_{n=0}^{N_t-1}\prod_{i\leq s}f_{i,n}g_jd\mu\]
(The limit exists since we have chosen the sequence $N_t$ to witness a particular limit point of the sequence.)

Finally, for each $s$, if $h$ is orthonormal to $(Z_s,\mathcal{B}_s,\mu_s)$, the set of $y$ such that $\int y_{s,0}hd\mu\neq 0$ has measure $0$, and so $\tilde f_s$ is measurable with respect to $(Z_s,\mathcal{B}_s,\mu_s)$.
\end{proof}

\begin{definition}
  If $K,H$ are finite graphs. define $t(K,H)$ to be the fraction of embeddings $\pi:|K|\rightarrow |H|$ such that $\pi$ is a graph embedding.
\end{definition}

\begin{corollary}[\cite{ES},\cite{Tao}]
  Let $H_n:=(V_n,E_n)$ be a sequence of finite graphs.  Then there are a measure space $(Y,\mathcal{B},\nu)$ and, for every finite graph $K:=(V,E)$, a function $\tilde V$ on $Y$ such that 
\[\liminf_{N\rightarrow\infty}\frac{1}{N}\sum_{n=1}^N t(K,H_n)\leq\int\tilde Vd\nu\leq\limsup_{N\rightarrow\infty}\frac{1}{N}\sum_{n=1}^N t(K,H_n)\]
\end{corollary}
\begin{proof}
  Let $S=G=\mathbb{N}$.  Let $X$ be the space of finite graphs, viewed as functions from finite subsets of $\mathbb{N}$ to $\{0,1\}$.  For any graph $K$, the function $u_K(H):=t(K,H)$ is continuous, so the result follows from Theorem \ref{main}.
\end{proof}

Note that a the sequence $(V_n,E_n)$ is convergent, in the sense of Elek and Szegedy, just if $t(K,H_n)$ converges for each $K$, in which case the $\liminf$ and $\limsup$ of the averages will converge to the same value.

\begin{corollary}
  Let $(\mathbb{N},E)$ be a countable graph.  Then there is a measure space $(Y,\mathcal{B},\mu)$ such that for any finite graph $V$, there is a measurable function $\tilde V$ such that
\[\liminf_{N\rightarrow\infty}\frac{1}{N}\sum_{n=1}^N t(V,([0,n],E\upharpoonright[0,n]))\leq\int\tilde Vd\mu\leq\limsup_{N\rightarrow\infty}\frac{1}{N}\sum_{n=1}^N t(V,([1,n],E\upharpoonright[1,n]))\]
\end{corollary}
\begin{proof}
  Let $S$ be the space $\mathbb{N}^{[2]}$ (that is, the set of pairs of distinct elements $n,m$), and $G$ the permutations on finite subsets of $\mathbb{N}$, acting on $S$ by $s\cdot \{a,b\}=\{s(a),s(b)\}$.  Take $I_n$ to be the set of permutations on $[1,n]$, and note that this is a F{\o}lner sequence.  Let $X=\{0,1\}$; then the characteristic function of $E$ is a function from $S$ to $X$.  The measure space $(Y,\mathcal{B},\mu)$ and functions $\tilde V$ exist by the main theorem.
\end{proof}

\section{Furstenberg-Style Proof}
Let $S,G,X, E:S\rightarrow X$, and $\{I_n\}$ be given as in the statement of Theorem \ref{main}.   Let $Y$ be the space of functions from $S$ to $X$.  Let $\mathcal{O}$ be a countable subbasis for $X$.  The product topology is on $Y$ is compact by the Tychonoff theorem, and is generated by sets of the form $\{y\mid y(s)\in U\}$ for elements $U\in\mathcal{O}$.  Call sets of this form and complements of such sets \emph{simple}.

Let $\mathcal{C}$ be the algebra generated from the simple sets by finite unions and intersections (the simple sets are already closed under complements).  This algebra is countable, so by diagonalizing, choose a subsequence $n_t\rightarrow\infty$ such that for every $C\in\mathcal{C}$, the limit
\[\rho(C):=\lim_{t\rightarrow\infty}\frac{1}{|I_{n_t}|}\sum_{g\in I_{n_t}}\chi_C(E\circ g)\]
is defined.

\begin{lemma}
  $\rho$ is finitely additive.
\end{lemma}
\begin{proof}
  Immediate by expanding the definition, since finite sums distribute over limits and multiplication by constants.
\end{proof}

Since $\rho$ is non-negative, it is also monotonic.  Define a $G$ action on $Y$ by $T_g(y)(s):=y(gs)$.  For any $g$ and large enough $n_t$, $\frac{|I_{n_t}\bigtriangleup g\cdot I_{n_t}|}{|I_{n_t}|}\rightarrow 0$, so
\[|\rho(C)-\rho(T_g C)|=\lim_{t\rightarrow\infty}\frac{1}{|I_{n_t}|}\sum_{g\in I_{n_t}\bigtriangleup g\cdot I_{n_t}}\chi_C(E\circ g)=0\]
For an open set $C\in\mathcal{C}$, define $\nu(C)$ to be the supremum of $\rho(C')$ where $C'$ ranges over closed elements of $\mathcal{C}$ contained in $C$:
\[\nu(C)=\sup_{D\subseteq C,D\text{ closed}}\rho(D)\]
$\nu$ is finitely additive on open sets, since if $A$ and $B$ are disjoint, $\nu(A\cup B)=\nu(A)+\nu(B)$ by the definition.  Therefore there is a unique finitely additive extension of $\nu$ to all of $\mathcal{C}$.

\begin{lemma}
  If $C$ is open then $\nu(C)\leq\rho(C)$.
\end{lemma}
\begin{proof}
  If $D\subseteq C$ then $\rho(D)\leq\rho(C)$.  Since
\[\nu(C)=\sup_{D\subseteq C,D\text{ closed}}\rho(D)\]
also $\nu(C)\leq\rho(C)$.
\end{proof}

\begin{lemma}
  For any $C\in\mathcal{C}$ and any $\epsilon>0$, there is a $D\subseteq C$ such that $D\in\mathcal{C}$, $D$ is closed, and $\nu(C)-\nu(D)<\epsilon$.
\end{lemma}
\begin{proof}
  First, suppose $C$ is simple.  If $C$ is closed, $C$ itself suffices, so suppose $C$ is open.  Then there is a closed $D\subseteq C$ such that $\rho(D)\geq\nu(C)-\epsilon$.  Since $\rho(D)\leq\nu(D)$, also $\nu(D)\geq\nu(C)-\epsilon$.

Any $C$ may be written as a finite union of finite intersections $C=\bigcup_{i\leq k}\bigcap_{j\leq m_i}C_{i,j}$ where each $C_{i,j}$ is simple; by adding additional terms, $\bigcap_j C_{i,j}$ may be assumed to be disjoint from $\bigcap_j C_{i',j}$ for $i'\neq i$.  Then it suffices to show the lemma for each $i\leq k$ separately.  Suppose that whenever $C=\bigcap_{j\leq m}C_j\cap D$ and $D$ is open then we may find a closed $D'\subseteq D$ such that $\nu(\bigcap C_j\cap D)-\nu(\bigcap C_j\cap D')<\epsilon$.  Then we may define $C'_j$ inductively so that if $C_j$ is closed, $C'_j:=C_j$, and if $C_j$ is open, $C'_j\subseteq C_j$, $C'_j$ is closed, and
\[\nu(\bigcap_{j'<j} C'_{j'}\cap\bigcap_{j'>j}C_{j'}\cap C_j)-\nu(\bigcap_{j'<j} C'_{j'}\cap\bigcap_{j'>j}C_{j'}\cap C'_j)<\epsilon/j\]
Then $\bigcap C'_j$ is closed and $\nu(\bigcap C_j)-\nu(\bigcap C'_j)<\epsilon$.

To show the assumption, choose $D'\subseteq D$ so that $\nu(D)-\nu(D')<\epsilon$ and write
\[C=\left(\bigcap C_i\cap D'\right)\cup\left(\bigcap C_i\cap (D\setminus D')\right)\]
Since $\bigcap C_i\cap (D\setminus D')\subseteq D\setminus D'$, it follows that $\nu(\bigcap C_i\cap (D\setminus D'))<\epsilon$, and therefore $\nu(C)-\nu(\bigcap C_i\cap D')<\epsilon$.
\end{proof}

By taking complements, for any $C\in\mathcal{C}$ and any $\epsilon>0$, there is a $D\supseteq C$ such that $D\in\mathcal{C}$, $D$ is open, and $\nu(D)-\nu(C)<\epsilon$.

\begin{lemma}
  $\nu$ is $\sigma$-additive.
\end{lemma}
\begin{proof}
  Suppose $\bigcup_{i\in\mathbb{N}}C_i=C$.  For any $\epsilon>0$, we may choose $C'\subseteq C$ such that $\nu(C)-\nu(C')<\epsilon/2$ and $C'$ is closed, and sets $C'_i\supseteq C_i$ such that $C'_i$ is open and $\sum_i (\nu(C'_i)-\nu(C_i))<\epsilon/2$.  $C'\subseteq \bigcup_i C'_i$, so by the compactness of $Y$, there is a finite subcover of $C'$, $C'\subseteq\bigcup_{j\leq k}C'_{i_j}$, so by finite additivity $\nu(C')\leq\sum_j \nu(C'_{i_j})$.  But then $\nu(C)\leq\sum_j\nu(C_{i_j})+\epsilon$.  Since we may choose $\epsilon$ arbitrarily small, it follows that $\nu(C)\leq\sum_i\nu(C_i)$.

  Conversely, since for each $k$, $\bigcup_{i\leq k}C_i\subseteq C$, finite additivity gives $\sum_{i\leq k}\nu(C_i)\leq\nu(C)$, and therefore $\sum_i\nu(C_i)\leq\nu(C)$.
\end{proof}

For each $s\in S$, define $\tilde E_s$ to be the function $y\mapsto y(s)$.  These functions are measurable since for any Borel set $B$ on $X$, $\{y\mid y(s)\in B\}$ belongs to $\mathcal{C}$; indeed, replacing $B$ by an arbitrary open set $U$ and applying the same argument shows that these functions are continuous.  Then by definition, $\tilde E_{gs}=\tilde E_s\circ T_g$.

Let some $u, s_1,\ldots,s_k$ be given.  Since $X^k$ is compact and $u$ is continuous, and a continuous function with compact support is integrable, it follows that the function $\tilde u$ given by $\tilde u(y):= u(\tilde E_{s_1}(x),\ldots,\tilde E_{s_k}(x))$ is integrable.  In particular, $|\tilde u(y)|$ has a compact range, and is therefore bounded by some $B$.

Therefore there is a sequence of functions $u_n$ of the form
\[u_n=\sum_{i=0}^{N_n}v_{n,i}\chi_{S_{n,i}}\]
for elements $v_{n,i}\in V$ and measurable sets $S_{n,i}$ such that
\[\lim_{N\rightarrow\infty}|u_n(y)-\tilde u(y)|=0\]
almost everywhere.  Since $|v_{n,i}|$ is bounded by $B$, for any $\epsilon>0$ we may choose $n$ so that
\[u_n=\sum_{i=0}^{N}v_i\chi_{S_i}\]
and there is a set $C$ so that $C\cup\bigcup_i S_i=Y$, $C\cap \bigcup_i S_i=\emptyset$, the $S_i$ are pairwise disjoint, $\nu(C)<\epsilon/2B$, and $|u_n(y)-\tilde u(y)|<\epsilon/2B$ whenever $y\not\in C$.

For each $i$, we may choose an open set $S'_i$ containing $S_i$ such that $\nu(S'_i)-\nu(S_i)<\epsilon/2|v_i|$; since the set of $v'\in V$ within $\epsilon/|v_i|$ of $v_i$ is open and $\tilde u$ is continuous, we may further require that for every $y\in S'_i$, $|\tilde u(y)-v_i|<\epsilon/2|v_i|$.

We may then choose a closed $S''_i\subseteq S'_i$ such that $\nu(S'_i)-\rho(S''_i)<\epsilon/2N|v_i|$.  Then
\[|\sum_i v_i\nu(S_i)-\sum_i v_i\rho(S''_i)|<\epsilon/2\]
But also
\[\sum_i v_i\rho(S''_i)=\lim_{t\rightarrow\infty}\sum_i v_{i}\frac{|\{g\in I_{n_t}\mid E\circ g\in S''_{i}\}|}{|I_{n_t}|}=\lim_{t\rightarrow\infty}\frac{1}{|I_{n_t}|}\sum_i\sum_{g\in I_{n_t}} v_{i}\cdot \chi_{S''_{i}}(E\circ g)\]
Whenever $\chi_{S''_i}(E\circ g)=1$, it follows that
\[|u(E(gs_1),\ldots,E(gs_k))-v_{i}|<\epsilon/2|v_{i}|\]
Therefore this limit is within $\epsilon/2$ of
\[\lim_{t\rightarrow\infty}\frac{1}{|I_{n_t}|}\sum_i\sum_{g\in I_{n_t}} u(E(gs_1),\ldots,E(gs_k))\cdot \chi_{S''_{i}}(E\circ g)\]
But since $\lim_{t\rightarrow\infty}\frac{1}{|I_{n_t}|}\sum_i\sum_{g\in I_{n_t}}\chi_{S''_i}(E\circ g)\geq 1-\epsilon/2B$, it follows that the limit is within $\epsilon$ of
\[\lim_{t\rightarrow\infty}\frac{1}{|I_{n_t}|}\sum_{g\in I_{n_t}} u(E(gs_1),\ldots,E(gs_k))\]
Putting this together,
\[|\int \tilde ud\nu-\lim_{t\rightarrow\infty}\frac{1}{|I_{n_t}|}\sum_{g\in I_{n_t}} u(E(gs_1),\ldots,E(gs_k))|<2\epsilon\]
for all $\epsilon$.  Since $\{I_{n_t}\}$ is a subsequence of $\{I_t\}$, the result follows.

%
%

\section{Minimality}
We may further prove that the construction given in the previous section is the smallest such system up to the choices made in the construction.

\begin{theorem}
   Let $(Y,\mathcal{C},\nu)$ and $\{\tilde E_s\}$ satisfy the conclusion of the main theorem.  Let $(X,\mathcal{B},\mu)$ and $\{\tilde D_s\}$ be given by the Furstenberg-style proof in the previous section so that for every $u, s_1,\ldots,s_k$, and $\alpha$,
 \[\nu (\{y\mid u(\tilde E_{s_1}(y),\ldots,\tilde E_{s_k}(y))>\alpha\})=\mu(\{x\mid u(\tilde D_{s_1}(x),\ldots,\tilde D_{s_k}(x))>\alpha\})\]
 Then there is a measurable measure-preserving function $\pi:Y\rightarrow X$ such that $\tilde E_s=\tilde D_s\circ\pi$.
 \label{minimality}
 \end{theorem}
 \begin{proof}
   The second condition uniquely defines $\pi(y)\in S\rightarrow X$ by $\pi(y)(s):=\tilde E_s(y)$.  Observe that the inverse image of each simple set on $X$ is measurable since each $\tilde E_s$ is measurable.  

The measure space $(X,\mathcal{B},\mu)$ is generated by sets of the form
\[B_{u,s_1,\ldots,s_k,\alpha}:=\{x\mid u(\tilde D_{s_1}(x),\ldots,\tilde D_{s_k}(x))>\alpha\}\]
But then
\[\nu(\{y\mid u(\tilde D_{s_1}(\pi(y)),\ldots,\tilde D_{s_k}(\pi(y)))>\alpha\}=\nu(\{y\mid u(\tilde E_{s_1}(y),\ldots,\tilde E_{s_k}(y))>\alpha\})\]
which is equal to $\mu(B_{u,s_1,\ldots,s_k,\alpha})$ by assumption.

 \end{proof}

\section{Nonstandard Proof}

For a general reference on nonstandard analysis, see, for example, \cite{Goldblatt}.

Fix an ultrapower extension of a universe containing all objects given in the premises and their powersets.  The sequence $\{I_n\}$ is a sequence of subsets of $G$, so let $Y:=I_m$ for some nonstandard $m$ be a subset of $G^*$.  For any internal $B\subseteq Y$, let $\mu(B)$ be the standard part of $\frac{|B|}{|Y|}$.  By the Loeb measure construction, this extends to a true measure on the $\sigma$-algebra extending the set of internal subsets of $Y$.

For each $s\in S$, define 
\[E_s(g):=E^*(gs^*)\]
for every $g\in G^*$.  These functions are internal, and therefore measurable.  Define an action of $G$ on $Y$ by $T_g g':=g'g^*$.  Then $g E_s(g')=E_s(g'g^*)=E^*(g'g^*s^*)=E_{gs}(g')$, as required.  Then $E_s$ is a function from $Y$ to $X^*$.

For every $\epsilon$ and large enough $t$, $\frac{|I_t\bigtriangleup g\cdot I_t|}{|I_t|}<\epsilon$, so in particular, $\frac{|Y\bigtriangleup g\cdot Y|}{|Y|}$ is infinitesimal, and therefore $\mu(Y\setminus T_g Y)=0$.  Then, for any internal $B\subseteq Y$, $T_g B\subseteq T_g Y$, and so $\mu(T_g B)=\mu(B)$.

Letting $\mathcal{O}$ be the open sets of $X$.  For each $x\in X^*$, consider
\[\{U\in\mathcal{O}\mid x\not\in U^*\}\]
Then the starred versions of the complements of these $U$ have the finite intersection property, and therefore the complements of these $U$ have the finite intersection property (for any finite set of these $U$, since there exists an element in $X^*$ in all of them, by transfer, there also exists an element in $X$ in all of them).  By the compactness of $X$, it follows that there is an element $st(x)$ contained in all of these sets.

Now let $\tilde E_s:=st\circ E_s$.  If $u$ is a continuous function from $X^k\rightarrow \mathbb{R}$, $\{s_1,\ldots,s_k\}$ is a finite set then for every $\epsilon$ and all but finitely many $N$,
\[\frac{1}{|I_N|}\sum_{g\in I_N}u(E(g\cdot s_1),\ldots,E(g\cdot s_k))+\epsilon\geq\liminf_{N\rightarrow\infty}\frac{1}{|I_N|}\sum_{g\in I_N}u(E(g\cdot s_1),\ldots,E(g\cdot s_k))\]
and so, by transfer,
\[\frac{1}{|Y|}\sum_{g\in Y}u^*(E^*(g\cdot s_1),\ldots,E^*(g\cdot s_k))\geq\liminf_{N\rightarrow\infty}\frac{1}{|I_N|}\sum_{g\in I_N}u(E(g\cdot s_1),\ldots,E(g\cdot s_k))\]

Note that for any $x$ and any open set $U$ of $X$ containing $st(x)$, $U^*$ contains $x$, since otherwise the complement of $U^*$ would be a closed set containing $x$, and therefore also $st(x)$.  Since $X$ is compact and $u$ is continuous, for each $\epsilon>0$ and each $x_1,\ldots,x_k$, there is an open subset $U$ of $X^k$ containing $st(x_1),\ldots,st(x_k)$ so that $|u(x'_1,\ldots,x'_k)-u(st(x_1),\ldots,st(x_k))|<\alpha$ for each $x'_1,\ldots,x_k\in U$.  But then $U^*$ contains $x_1,\ldots,x_k$, which means that $|u^*(x_1,\ldots,x_k)-u(st(x_1),\ldots,st(x_k))|<\alpha$.  Since this holds for every $\alpha$, it follows that $st(u^*(x_1,\ldots,x_k))=u(st(x_1),\ldots,st(x_k))$.  Therefore
\[\frac{1}{|Y|}\sum_{g\in Y}u(\tilde E_{s_1}(g),\ldots,\tilde E_{s_k}(g))\geq\liminf_{N\rightarrow\infty}\frac{1}{|I_N|}\sum_{g\in I_N}u(E(g\cdot s_1),\ldots,E(g\cdot s_k))\]
and also
\[\int u(\tilde E_{s_1}(g),\ldots,\tilde E_{s_k}(g))d\mu=st(\frac{1}{|Y|}\sum_{g\in Y}u(\tilde E(g\cdot s_1),\ldots,\tilde E(g\cdot s_k)))\geq\liminf_{N\rightarrow\infty}\frac{1}{|I_N|}\sum_{g\in I_N}u(E(g\cdot s_1),\ldots,E(g\cdot s_k))\]
Since the dual claim holds for the $\limsup$, the result follows.

\section{Amenable Groups}
It is natural to wonder whether the requirement that $G$ be countable can be lifted.  Examination of the proofs above shows that the only place countability is used is in the diagonalization process in the Furstenberg-style proof.  Therefore we may immediately extend this to the following:
\begin{theorem}
  Let $S$ be set and let $G$ be an amenable semigroup acting on $S$ with left-invariant mean $m$.  Let $X$ be a compact space.  Let $E:S\rightarrow X$ be given.

  Then there are a dynamical system $(Y,\mathcal{B},\mu,(T_g)_{g\in G})$ and functions $\tilde E_{s}:Y\rightarrow X$ for each $s\in S$ such that the following hold:
  \begin{itemize}
  \item For any $g,s$, $\tilde E_{gs}=\tilde E_s\circ T_g$
  \item For any continuous function $u:X^k\rightarrow \mathbb{R}$ for some integer $k$, and any finite sequence $s_1,\ldots, s_k$
\begin{align*}
\liminf_{n\rightarrow\infty}\frac{1}{|I_n|}\sum_{g\in I_n}u(E(g s_1),\ldots,E(gs_k))&\leq\int  u(\tilde E_{s_1},\ldots,\tilde E_{s_k})d\mu\\
&\leq\limsup_{n\rightarrow\infty}\frac{1}{|I_n|}\sum_{g\in I_n}u(E( gs_1),\ldots,E( gs_k))
\end{align*}
\end{itemize}
\end{theorem}

The Furstenberg-style proof then requires that $\rho(C)$ be given by
\[\rho(C):=m(\{g\mid E\circ g\in C\})\]
and the rest of the proof goes through by replacing every occurrence of $\lim_{t\rightarrow\infty}\frac{|\{g\in I_{n_t}\mid \phi(g)|}{|I_{n_t}|}$ by $m(\{g\in G\mid\phi(g)\})$.  Similarly, in the non-standard proof we take $\mu(B)$ for internal sets $\{B_n\}$ to be given by
\[\mu(\{B_n\}):=m^*(B_n)\]

This is unsatisfying, however, since when $S$ is uncountable, we would generally expect it to have some sort of topological structure in its own right.  Specifically, if $E$ is a continuous function from $S$ to $X$, we would like a natural topology on $Y$ so that $\tilde E$ is a continuous function from $Y\times S$ to $X$.

Bergelson, Boshernitzan, and Bourgain \cite{BergelsonBB} give an example showing the correspondence fails for continuous $\mathbb{R}$ actions (even when $X$ is the discrete set $\{0,1\}$), so we may ask what additional conditions are needed.

In order to make the Furstenberg-style proof go through, it suffices to give a topology on the space $C(S,X)$ of continuous functions from $S$ to $X$ which makes the evaluation function continuous but is still compact.  When $S$ is locally compact, the smallest topology making the evaluation function continuous is the compact-open topology; we must then check that the closure of $\{g\circ E\mid g\in G\}$ is compact in this topology.  The Arzel\'a-Ascoli Theorem states that this occurs exactly if this set is evenly continuous; that is, for every $s\in S$, every $x\in X$, and any neighborhood $V$ of $x$, there are neighborhoods $U$ of $s$ and $U'$ of $x$ such that whenever $ g\circ E(s)\in U'$, $ g\circ E(s')\in U'$ for every $s'\in U$.  (The additional requirement of the theorem, pointwise boundedness, is automatically satisfied since $X$ is assumed to be compact.)

Therefore we obtain the following extension of the correspondence:
\begin{theorem}
  Let $S$ be a locally compact space and let $G$ be an amenable semigroup acting continuously on $S$ with left-invariant mean $m$.  Let $X$ be a compact space.  Let $E:S\rightarrow X$ be continuous.  Suppose that $\{ g\circ E\mid g\in G\}$ is evenly continuous.

  Then there are a dynamical system $(Y,\mathcal{B},\mu,(T_g)_{g\in G})$ and functions $\tilde E_{s}:Y\rightarrow X$ for each $s\in S$ such that the following hold:
  \begin{itemize}
  \item For any $g,s$, $\tilde E_{gs}=\tilde E_s\circ T_g$
  \item For any continuous function $u:X^k\rightarrow \mathbb{R}$ for some integer $k$, and any finite sequence $s_1,\ldots, s_k$
    \begin{align*}
      \liminf_{n\rightarrow\infty}\frac{1}{|I_n|}\sum_{g\in I_n}u(E(g s_1),\ldots,E(gs_k))&\leq\int  u(\tilde E_{s_1},\ldots,\tilde E_{s_k})d\mu\\
      &\leq\limsup_{n\rightarrow\infty}\frac{1}{|I_n|}\sum_{g\in I_n}u(E( gs_1),\ldots,E( gs_k))
    \end{align*}
  \item $\mathcal{B}$ consists of the Borel sets generated by a topology on $Y$ such that, under this topology, the functions $(g,y)\mapsto T_g y$ and $(s,y)\mapsto \tilde E_s(y)$ are continuous
  \end{itemize}
\end{theorem}

\bibliography{Norm}
\end{document}